\documentclass[11pt,a4paper]{article}

\input prepictex
\input pictex
\input postpictex

\usepackage[T1]{fontenc}
\usepackage[latin1]{inputenc}
\usepackage[francais,english]{babel}

\usepackage{verbatim}
\usepackage{amsfonts}
\usepackage{amsmath}
\usepackage{amssymb}
\usepackage{amsthm}
\usepackage{mathrsfs}
\usepackage{enumerate}
\usepackage{setspace}

\usepackage{times}
\usepackage[T1]{fontenc}
\usepackage{tocloft}
\usepackage[ruled,linesnumbered]{algorithm2e}

\newcommand{\CC}{\mathbb C}
\newcommand{\QQ}{\mathbb Q}
\newcommand{\PP}{\mathbb P}
\newcommand{\KK}{\mathbb K}

\newcommand{\ZZ}{\mathbb Z}
\newcommand{\NN}{\mathbb N}
\newcommand{\LL}{\mathbb L}

\newcommand{\C}{\mathscr W}

\newcommand{\FF}{\mathbb F}
\newcommand{\xx}{P}

\newtheorem{thm}{Theorem}[section]

\newtheorem{lemma}[thm]{Lemma}
\newtheorem{corollary}[thm]{Corollary}

\newtheorem{definition}[thm]{Definition}
\newtheorem{proposition}[thm]{Proposition}
\newtheorem{example}[thm]{Example}

\newtheorem{remark}[thm]{Remark}

\def\cocoa{{\hbox{\rm C\kern-.13em o\kern-.07em C\kern-.13em o\kern-.15em A}}}

\newcommand{\PREPROCESSING}{\textbf{Preprocessing:\ }}
\newcommand{\RETURN}{\textbf{Return:\ }}
\numberwithin{equation}{section}

\title{Modular absolute decomposition of equidimensional polynomial ideals}
\author{Cristina Bertone\thanks{Written with the support of the PRIN project ``'Geometria delle varieta' algebriche e dei loro spazi di moduli'' funds (co-financed by the MIUR, cofin 2008).}}

\begin{document}

\bibliographystyle{plain}

\maketitle

\begin{abstract}
In this paper, we present a modular strategy which describes key properties of the absolute primary decomposition of an equidimensional polynomial ideal defined by polynomials with rational coefficients. The algorithm we design is based on the classical technique of elimination of variables and colon ideals and uses a tricky choice of prime integers to work with. Thanks to this technique, we can obtain the number of absolute irreducible components, their degree, multiplicity and also the affine Hilbert function  of the reduced components (namely, their initial ideal w.r.t. a degree-compatible term ordering).
\end{abstract}

\section*{Introduction}

In this paper we design an algorithm whose aim is quite simple to state:\\
 Given a set of polynomial rational equations which define an equidimensional algebraic set $\C$ of $\CC^n$, we would like to get as many information as possible on the \emph{irreducible components} of this algebraic set.

We can rephrase the problem in algebraic language: given an ideal  (with suitable hypothesis on its dimension) in the polynomial ring over $\CC[\mathbf X]$ defined by rational generators, find all the possible information about its primary components.

\smallskip

The problem is really simple to state and many authors looked for efficient strategies  to get the irreducible decomposition of an algebraic set: one can see for instance \cite{Deck} and the re\-fe\-ren\-ces therein to have an overlook of the different techniques. In many Computer Algebra Systems (CAS for short) you can find routines computing the primary decomposition of an ideal: the underlying algorithm is often the one described in \cite{GTO}. Nevertheless, the problem is really challenging since the existing algorithms and implementations often focus on particular cases, e.g $0$-dimensional ideals (see for instance  \cite{cd}); for more general situations, also the best implemented algorithms (for instance, the ones in \cite{cct} or \cite{GTO}) may have unsatisfying time of execution  and there  may be problems of memory allocation. In fact, the computations required to a personal computer to find a primary decomposition are often quite heavy.\\
 Our aim is to design an algorithm concerning the decomposition of an ideal which can give an output in a reasonable time and with a limited use of memory.

\smallskip
 
The main computational tool that we use are modular computations, taking this technique from the absolute factorization algorithm for bivariate polynomials presented in \cite{BCG}.\\
  Recent papers about decomposition of algebraic sets (see for instance \cite{galrup}, \cite{svw}) focus on getting information about the irreducible components from a generic section with a linear space, namely they bring back the problem to the study of a $0$-dimensional ideal.\\
In this paper we will bring back the problem of computing a primary decomposition to the problem of computing an absolute factorization; this technique is in some sense "classical" (\cite{herm}), but not very exploited because not efficient from the computational point of view; a powerful improvement of this technique is in \cite{cct}, where the authors avoid the use of \emph{generic} projections in order to compute the equidimensional isoradical decomposition of an ideal, using as coefficient ring $\QQ$ or a finite field of positive characteristic. Our approach is instead to use generic projections (by  a generic change of variables and projections on coordinate linear spaces) and exploit modular computations
to move around the computational difficulties, preserving a lot of data concerning the absolute primary decomposition of the ideal.
 The output of our algorithm will not be the complete primary decomposition of the given ideal, but it will  return information concerning the components, such as number, degree, multiplicity and, for reduced components, the affine Hilbert function.
 
 \smallskip

In Section \ref{defgen} we will show that once known information about one of the primary components of the ideal, the same is known for other components too. We simply rephrase the definition of "conjugacy" for absolute factors of a multivariate polynomial with rational coefficients (see \cite{CG1}, Lemma 9.0.8), for primary components of an ideal generated by polynomials with rational coefficients. Degree, multiplicity and affine Hilbert function  are "invariant by conjugacy", so  if we obtain this information about a primary component, we actually have the same information for all the primary components in its "conjugacy class", avoiding to repeat computations.

In Section \ref{sectgoodred} we show that, fixed an algebraic extension $\LL$ of $\QQ$, there are infinite prime integers that implicitly define a homomorphism from $\LL$ to $\ZZ/p\ZZ$ (more precisely, an inclusion of $\LL$ in $\QQ_p$, Lemma \ref{rootinQp}). This means that with a careful choice of a prime $p$, we can reduce the coefficients of a polynomial in $\LL[\mathbf X]$ modulo $p$. Furthermore, infinite prime numbers preserve interesting properties of an ideal in $\LL[\mathbf X]$, namely the initial ideal with respect to some degree-compatible term ordering and, as a consequence of this, the affine Hilbert function.\\
Summing up, we can choose a prime $p$ which allows modular computations in $\LL$ (we can choose it using Lemma \ref{rootinQp}). Only a finite number of primes $p$ does not preserve the properties of the primary components we are interested in, so we can assume that we are avoiding them by taking a ``generic'' prime $p$.

In Section \ref{exactstrat} we present the exact strategy to obtain the prime components of an ideal $\mathfrak a$. This technique is mainly based on elimination of variables, in order to bring back the problem of primary decomposition to a problem of factorization. This strategy was first investigated by Grete Hermann  in \cite{herm} and it is similar to the splitting techniques presented in \cite{cct}, but we present it completely for lack of an accessible reference on the whole strategy.
Nevertheless, the technique of Section \ref{exactstrat} is not efficient from a computational point of view: first of all, projections are actually computed with a generic change of coordinates and an elimination of variables performed by a Groebner Basis; then, in order to obtain the reduced primary components, we compute a colon ideal; this  is performed again by an elimination Groebner Basis.

In Section \ref{compmodp} we try to gain in computational efficiency, even if we ``lose'' the exactness of Section \ref{exactstrat}.  We will apply the modular results of Section \ref{sectgoodred} on the exact algorithms of Section \ref {exactstrat}. We compute projections, factorizations and colon ideals modulo well-chosen prime integers; we do not get the reduced primary components, but we obtain an algorithm (Algorithm \ref{algdecmodp}) which can compute the initial ideal of the reduced components of $\mathfrak a$ and give information about the non-reduced components.

Finally, in Section \ref{tricks}, we test our strategy on a  simple example, a complete intersection ideal in 3 variables, getting the output of Algorithm \ref{algdecmodp}  in a really reasonable time. The same ideal could not be decomposed by other CAS in 1 hour (because of problems with memory allocation); obviously the comparison between our strategy and implemented primary decomposition algorithms is not complete, since Algorithm \ref{algdecmodp} does not return the absolute primary decomposition of the input ideal. However this comparison enlightens promising performances of our strategy and this can be a starting point for designing an efficient primary decomposition algorithm.

\subsection*{Notations}

In what follows, we will work in a polynomial ring $R$ with coefficients in a field $\KK$ of characteristic 0: $R=\KK[X_1,\dots,X_n]=\KK[\mathbf X]$. We will precise, when needed, if $\KK=\QQ$, $\QQ(\alpha)$ or $\CC$.

Given an ideal $\mathfrak a\subseteq R$ we will consider its zero set in $\CC^n$: $\C=V(\mathfrak a)=\{P\in \CC^n|f(P)=0 \forall f \in \mathfrak a\}$.

\section{Affine Hilbert Function and Conjugacy}
\label{defgen}

In this section we introduce the main definitions concerning the primary decomposition  and the affine Hilbert function of an ideal $\mathfrak a \subseteq R$.\\
 We show that  some of the primary components of an ideal $\mathfrak a$ are very ``similar'' to each other, in the sense that given a set of generators for a primary component, we can get a set of generators for another primary component  by means of \emph{conjugacy}, just like we do for the absolute factors of a polynomial with rational coefficients (\cite{CG1}, Lemma 9.0.8).
This allows us to avoid repeating the computation of the affine Hilbert function for the conjugate components, since it is invariant by conjugacy.

For all the definitions and properties concerning primary decomposition, the main reference is \cite{atiyah}, Chapter 4.

\begin{definition}
A proper ideal $\mathfrak q$ in a ring $R$ is \emph{primary} if the following condition holds:
\[
xy \in \mathfrak q \text{ and } x \notin \mathfrak q \Rightarrow y\in \sqrt{\mathfrak q}.
\]
\end{definition}

\noindent Every prime ideal is obviously primary.

\begin{proposition}[\cite{atiyah}, Proposition 4.1]
Let $\mathfrak q$ be a primary ideal in $R$. Then $\mathfrak p=\sqrt{\mathfrak q}$ is the smallest prime ideal containing $\mathfrak q$; we say that $\mathfrak q$ is $\mathfrak p$-primary.
\end{proposition}

\begin{definition}\label{primdecdef}
A \emph{primary decomposition} of an ideal $\mathfrak a$ in $R$ is an expression of $\mathfrak a$ as a finite 
intersection of primary ideals:
\begin{equation}\label{primarydef}
\mathfrak a= \bigcap_{i=1}^r \mathfrak q_i. 
\end{equation}
 If moreover 
 \begin{enumerate}
 \item\label{primacondprim} $\mathfrak q_i \not \supseteq\bigcap_{i\neq j}\mathfrak q_j$;
\item\label{seccondprim} the prime ideals $\mathfrak p_i=\sqrt{\mathfrak q_i}$ are all distinct,
 \end{enumerate} 
 then the primary decomposition (\ref{primarydef}) is said to be \emph{minimal}. 
Any primary decomposition can be reduced to a minimal one (see \cite{atiyah}, page 52).
\end{definition}

Since we assume to work in a polynomial ring $R$ with coefficients in a field, a minimal primary decomposition always exists.

\smallskip

The factorization of a multivariate polynomial and the primary decomposition of a polynomial ideal are very close to each other: indeed,  the primary decomposition of a principal ideal corresponds to computing the absolute factorization of the generator of the ideal. So we can look at the factorization of a multivariate polynomial as a particular case of primary decomposition.\\
 Thanks to this similarity, it is natural to extend the definition of \emph{degree} and \emph{multiplicity} of a factor to a primary component. 
We can define them through the \emph{affine Hilbert function} (\cite{Robb2}, Section 5.6):

\begin{definition}\label{affHfunction}
 Let  $\mathfrak a$ be an ideal in the polynomial ring $R$ standard graded.

We first define  $\langle R_{\leq i}\rangle$, the vector space generated by all the polynomials of $R$ of degree $\leq i$. The $\KK$-vector space $\langle \mathfrak a_{\leq i}\rangle$ is the vector subspace of $\langle R_{\leq i}\rangle$ which consists of the polynomials of $\mathfrak a$ of degree $\leq i$. Since 
$\mathfrak a_{\leq i}=R_{\leq i}\cap \mathfrak a$, we can view the vector space $R_{\leq i}/\mathfrak a_{\leq i}$ as a vector subspace of $R/\mathfrak a$. 

The map $HF^a_{R/\mathfrak a} : \ZZ \rightarrow \ZZ$ defined by 
\[
HF^a_{R/\mathfrak a} (i) =\dim_\KK(
\langle R_{\leq i}\rangle /\langle \mathfrak a_{\leq i}\rangle) \quad \text{for } i\in \ZZ
\]
is called the \emph{affine Hilbert function} of $R/\mathfrak a$.
\end{definition}

From Definition \ref{affHfunction}, it is natural to define the affine Hilbert series, polynomial, dimension and the affine regularity index of $R/\mathfrak a$. The definitions are  similar to the analogous  for the homogeneous case; for all these definitions and their properties, we refer to \cite{Robb2}, Section 5.6. 

\begin{definition}\label{Hdeg}
 Let $\mathfrak a$ be a proper ideal in $R$, consider its affine Hilbert polynomial $HP^a_{R/\mathfrak a}(t)\in\QQ[t]$. 
The \emph{degree} of $R/\mathfrak a$ is $(\dim(R/\mathfrak a))!\cdot(\text{lcoeff}(HP^a_{R/\mathfrak a}(t))$.
\end{definition}

We will often say ``dimension and degree of $\mathfrak a$'', meaning the dimension and degree of  $R/\mathfrak a$.

Finally, once defined the degree of an ideal, we can define the multiplicity of a primary component. Here we state the algebraic definition, which corresponds to the intuitive idea that the mul\-ti\-pli\-ci\-ty counts ``how many times'' the primary component is repeated.

\begin{definition}\label{Hmult}[\cite{Bates}, Definition 10]
 Let $\mathfrak q\in R$ be a $\mathfrak p$-primary ideal. Then the \emph{multiplicity} of $\mathfrak q$, $\mathrm{mult}(\mathfrak q_i)$, in $\mathfrak p$ is $\deg(\mathfrak p)/\deg(\mathfrak q)$.
\end{definition}

We will often talk about the multiplicity of a primary component, implying that it is the multiplicity in its radical.

We now briefly recall how to explicitly compute the affine Hilbert function. 

For a polynomial ring $\KK[\mathbf X]$, we will denote with $\mathbb T^n$ the monoid of monomials in $\KK[\mathbf X]$ and with $\mathbf X^I=X_1^{i_1}\cdots,X_n^{i_n}, i_j\in \NN$ a monomial.
A term ordering $\preceq$ on $\mathbb T^n$ is \emph{degree compatible} if for any couple of monomials $\mathbf X^{I}, \mathbf X^J$
\[ \mathbf X^I\preceq \mathbf X^J\Rightarrow \deg\mathbf X^I\leq \deg \mathbf X^J.\]

Once fixed a term ordering $\preceq$ on $\mathbb T^n$, for a polynomial $g\in \KK[\mathbf X]$, we denote with $LM_\preceq(g)$ (or simply $LM(g)$ if there is no ambiguity) the maximal  monomial with respect to $\preceq$ appearing in $g$ with non-zero coefficient.\\ 
In the following Proposition, $HF_{R/\mathfrak b}$ is the Hilbert function for a homogeneous ideal.

\begin{proposition}\label{propositionHAff}
Let $\preceq$ be a degree compatible term ordering on $\mathbb T^n$.
 For every $i\in\ZZ$, we have $HF^a_{R/\mathfrak a}(i)=\sum_{j=0}^iHF_{R/ LM_\preceq(\mathfrak a)}(j)$. In particular,
we have $HF^a_{R/\mathfrak a}(i)=HF^a_{R/LM_\preceq(\mathfrak a)}(i)$ for all $i\in\ZZ$.
\end{proposition}
\begin{proof}
 See \cite{Robb2}, Proposition 5.6.3.
\end{proof}

Proposition \ref{propositionHAff} gives us the practical way to compute the affine Hilbert function of $\mathfrak a$: chosen a degree compatible term ordering $\preceq$, we can compute the initial ideal of $\mathfrak a$ 
and then we count the number of elements in the vector space $R/ LM_\preceq(\mathfrak a)(j)$  for every $j\leq i$.

\medskip

We now show with a few lemmas that, given an ideal $\mathfrak a$ defined by polynomials with rational coefficients, the relation of the primary decomposition on $\QQ[\mathbf X]$ and the primary decomposition on $\CC[\mathbf X]$ is similar to the relation between a rational and an absolute factorization of a multivariate polynomial with rational coefficients  (as shown in \cite{CG1}, Lemma 9.0.8). In other words,  there is a conjugacy relation among some of the primary components.
\begin{definition}\label{nonrat}
 Consider $\mathfrak a$ ideal in $\CC[\mathbf X]$ defined by a set of polynomials with rational coefficients.\\ 
Let $\mathfrak a=\bigcap_{j=1}^s\mathfrak q_i, \mathfrak q_i\in \QQ[\mathbf X]$  be the \emph{rational primary decomposition} of $\mathfrak a$; $\mathfrak q_i$ (resp. $V(\mathfrak q_i)$) is a \emph{rational primary component} of $\mathfrak a$ (resp. of $V(\mathfrak a)$).\\
  We then consider a primary decomposition in $\CC[\mathbf X]$ of each rational primary component $\mathfrak q_i$: $$\mathfrak q_i=\bigcap_{j=1}^{r_i}\mathfrak q_i^{(j)}\subseteq\CC[\mathbf X].$$
  
If $r_i=1$, we say that $\mathfrak q_i$ (resp. $V(\mathfrak q_i$)) is a \emph{pure rational} component of $\mathfrak a$ (resp. of $V(\mathfrak a)$).
\end{definition}


Consider a  non-pure rational component $\mathfrak q_i$ of an ideal $\mathfrak a\subseteq R$. Let $\LL_i$ be the smallest (w.r.t. the degree of extension on $\QQ$) normal algebraic extension  of $\QQ$ such that $\mathfrak q_i^{(1)}$ has a set of generators in $\LL_i[\mathbf X]$; assume that $\LL_i=\QQ(\alpha_i)$ and we denote  with $f_j(\alpha_i, \mathbf X)$ a polynomial in the chosen set of generators of $\mathfrak q_i$; indeed, we can think of such a generator as a polynomial in $\QQ[Z,\mathbf X]$ with $Z$ evaluated in $\alpha_i$. Consider the Galois group of $\LL_i$ over $\QQ$, $\mathcal G_i=Gal(\LL_i/\QQ)$.

For every $\sigma \in \mathcal G_i$, starting from $\mathfrak q_i^{(1)}$, with $\sqrt{\mathfrak q_i^{(1)}}=\mathfrak p_i^{(1)}$,  we can define an ideal in the following way
\[
 \mathfrak q_i^{(1)}=(f_1(\alpha_i,\mathbf X),\dots,f_l(\alpha_i,\mathbf X))\rightarrow
  \sigma(\mathfrak q_i^{(1)})=(f_1(\sigma(\alpha_i),\mathbf X),\dots,f_l(\sigma(\alpha_i),\mathbf X)).
\]
Obviously, the definition of $\sigma (\mathfrak q_i^{(1)})$ is independent from the chosen set of generators of $\mathfrak q_i^{(1)}$; $\mathfrak q_i^{(1)}$ and $\sigma (\mathfrak q_i^{(1)})$ have the same dimension; furthermore it is straightforward that the ideal $\sigma(\mathfrak q_i^{(1)})$ is $\sigma(\mathfrak p_i^{(1)})$-primary. Finally, if $\tau,\sigma \in \mathcal G_i$, $\tau\neq \sigma$, then $\tau (\mathfrak q_i^{(1)})\neq\sigma (\mathfrak q_i^{(1)})$ and $\tau (\mathfrak p_i^{(1)})\neq\sigma (\mathfrak p_i^{(1)})$.

\noindent We now show that actually the ideals $\sigma(\mathfrak q^{(1)}_i)$ are the primary components of $\mathfrak q_i$ in $\overline \QQ[\mathbf X]$.

\begin{lemma}\label{galprimdec}
Consider $\mathfrak q_i$  a non-pure rational component of $\mathfrak a	\subseteq \CC[\mathbf X]$ defined by polynomials with rational coefficients, $\LL_i$ the smallest normal algebraic extension  of $\QQ$ such that $\mathfrak q_i^{(1)}$ has a set of generators in $\LL_i[\mathbf X]$, $\mathcal G=Gal(\LL_i/\QQ)$. The minimal primary decomposition of $\mathfrak q_i$ 
is
\begin{equation}\label{primgal}
\mathfrak q_i=\bigcap_{\sigma \in \mathcal G_i}\sigma\left( {\mathfrak q_i^{(1)}}\right),
\end{equation}
and in particular $r_i=\#\mathcal G_i=[\LL_i:\QQ]$.
\end{lemma}

\begin{proof}
The ideal $\sigma(\mathfrak q_{i}^{(1)})$ is $\sigma(\mathfrak p_{i}^{(1)})$-primary. So $\bigcap_{\sigma \in \mathcal G}\sigma\left( {\mathfrak q_i^{(1)}}\right)$ is a primary decomposition of an ideal $\mathfrak b$. 

Furthermore, it is a minimal primary decomposition.  Indeed, thanks to the definition of the ideals through the automorphism of $\LL_i$, all the associated primes $\sigma_i(p_i^{(1)})$ are distinct; for what concerns redundant primary components, for any $\sigma \in \mathcal G$,  since $\LL_i$ is the minimal normal algebraic extension containing a set of generators of $\mathfrak q_i^{(1)}$, then there is $f\in \mathfrak q_i^{(1)}$ such that $\prod_{\tau\neq\sigma}\tau(f)$ is not in $\sigma(\mathfrak q_i^{(1)})$.

Using \cite{atiyah}, Exercises 12 and 13 of Chapter 5, if we consider the associated primes and the natural homomorphism $\QQ[\mathbf X]\rightarrow \LL_i[\mathbf X]$, then the set of prime ideals $\{\mathfrak p_{ij}\}$ is the same as the set of prime ideals of $\LL_i[\mathbf X]$ whose contraction is $\mathfrak p_i$. Then $\mathcal G$ acts transitively on the set $\{p_{i}^{(j)}\}_{j=1,\dots,r_i}$, that is $\{p_{i}^{(j)}\}_{j=1,\dots,r_i}=\{\sigma(\mathfrak p_i^{(1)})\}_{\sigma \in \mathcal G}$.

So $\bigcap_{\sigma \in \mathcal G}\sigma\left( {\mathfrak q_i^{(1)}}\right)$ is a minimal primary decomposition of $\mathfrak q_i$; since $\mathfrak q_i$ is a primary ideal, all its primary components in $\LL_i[\mathbf X]$  have same dimension and there are no embedded components, so its primary decomposition is unique and is exactly the one in (\ref{primgal}).
\end{proof}

\begin{lemma}\label{deccong}
Consider $\mathfrak a\subseteq \CC[\mathbf X]$, defined by a set of polynomials with rational coefficients, with no embedded components. 
Then the minimal primary decomposition of $\mathfrak a$ is
\begin{equation}\label{decprimfin}
\mathfrak a=\bigcap_{i=1}^{r}\left(\bigcap_{\sigma \in \mathcal G_i}\mathfrak \sigma\left(q_{i}^{(1)}\right)\right).
\end{equation}
\end{lemma} 
\begin{proof}
Since the ideals  $\sigma\left(q_{i}^{(1)}\right)$ are primary, we just need to show that the decomposition is minimal.

Condition \ref{primacondprim} of Definition \ref{primdecdef} about minimality is straightforward from Lemma \ref{galprimdec}.

For what concerns Condition \ref{seccondprim}, we just have to point out that if there is $\tilde{\mathfrak p}$ 
 associated to two different primary components of $\mathfrak a$ in $\CC[\mathbf X]$, $\sigma\left(q_{i}^{(1)}\right)$, $\sigma \in \mathcal G_i$, and $\tau\left(q_{j}^{(1)}\right)$, $\tau \in \mathcal G_j$,  then we would have two associated primes of the rational primary decomposition included in one other or equal. But this contradicts the fact that $\mathfrak a$ has no embedded components and the minimality of the rational primary decomposition of $\mathfrak a$.
\end{proof}

Up to relabeling the automorphisms of $Gal(\LL_i/\QQ)$, we can rewrite (\ref{decprimfin}) as
\begin{equation}\label{relabgaldec}
 \mathfrak a=\bigcap_{i=1}^{r}\left(\bigcap_{j=1}^{r_i}\mathfrak q_{i}^{(j)}\right),
\end{equation}
with $r_i=[\LL_i:\QQ]$, $\mathfrak q_i^{(j)}=\sigma_j\left(\mathfrak q_i^{(1)}\right)$.

\begin{definition}
The primary decomposition (\ref{relabgaldec}) is the \emph{absolute  primary decomposition} of $\mathfrak a$  and for $i$ such that $r_i\geq 2$, we say that $\mathfrak q_i^{(j)}$ (resp. $V(\mathfrak q_i^{(j)})$) is an \emph{absolute component} of $\mathfrak a$ (resp. of $V(\mathfrak a)$).
\end{definition}

If $\mathfrak q_i=\bigcap_{j=1}^{r_i}\mathfrak q_{i}^{(j)}$, we say that $\mathfrak q_i^{(j)}$ and $\mathfrak q_i^{(j')}$ are \emph{conjugate}, and that $\{\mathfrak q_i^{(j)}\}_{j=1,\dots,r_i}$ is a \emph{conjugacy class}. Any number or property of an absolute component is \emph{invariant by conjugacy} if it is the same for all the absolute components in the same conjugacy class.


\medskip

From now on we will focus on a particular kind of ideals, \emph{equidimensional} ones. 

 \begin{definition}
An ideal $\mathfrak a$ (resp. an algebraic set $\C \subseteq \CC^n$) is \emph{equidimensional} if all of its primary components (resp. all of its irreducible components) have the same dimension.
 \end{definition}

Thanks to \cite{atiyah}, Corollary 4.11, if $\mathfrak a$ is equidimensional, all of its primary components are uniquely determined and so in this case the primary decomposition is unique.

Furthermore, there is a wide class of ideals which are equidimensional: if we consider a complete intersection  ideal $\mathfrak a \in R$ generated by $n-c$ polynomials, it is equidimensional. This can be seen as a consequence of the Affine Dimension Theorem (\cite{hart}, Chapter I, Proposition 7.1).

We can finally fix our  purpose.

\begin{quotation}
 Given a non-prime equidimensional ideal $\mathfrak a\subseteq \CC[\mathbf X]$, generated by polynomials with rational coefficients, we write its primary de\-com\-po\-si\-tion as in (\ref{relabgaldec}). Then for every rational primary component $\mathfrak q_i$ of $\mathfrak a$, we  would like to find the numbers $r_i$, $\deg (\mathfrak q_i^{(1)})$and $\mathrm{mult}(\mathfrak q_i^{(1)})$.\\
 If $\mathrm{mult}(\mathfrak q_i^{(1)})=1$ (the primary component is reduced, and so it is prime) then we would also like to compute the affine Hilbert function of $R/\mathfrak q_i^{(1)}$.
\end{quotation}

\begin{remark}
Thanks to Lemma \ref{deccong}, all the information concerning the primary component $\mathfrak q_i^{(1)}$ (such as degree, multiplicity and affine Hilbert function if the component is prime) are the same for all the conjugate components $\sigma(\mathfrak q_i^{(1)})$, $\sigma \in Gal(\LL_i/\QQ)$, since we actually compute them by an initial ideal, which is invariant by conjugacy.
\end{remark}


\section{Algebraic extensions of $\QQ$ and modular computations}\label{sectgoodred}

We are interested in preserving some properties of an ideal $\mathfrak a$ in $R=\LL[\mathbf X]$ ($\LL$ is a normal algebraic extension of $\QQ$) ``modulo'' a well-chosen prime integer $p$. First of all we need to establish how we can compute an algebraic number modulo a prime $p$ and then we will see that in general the reduction modulo $p$ of the coefficients of a polynomial ideal preserves the affine Hilbert function of the ideal itself.

\medskip

Let $\mathfrak a=(f_1,\dots,f_s)\subseteq\LL[\mathbf X]$ be an ideal. 
$\LL$ is a normal algebraic extension of $\QQ$ of degree $s$: $\LL\simeq \QQ(\alpha)$, where $\alpha$ is an algebraic number such that its minimal polynomial is $q(T)\in\QQ[T]$, $\deg q(T)=s$ and $q(T)=\sum_{i=1}^{s}(T-\sigma_i(\alpha))$ where $\sigma_i$ are the automorphism of $\LL$ fixing $\QQ$, $\sigma_i(\alpha)=\alpha_i$ are the conjugates of $\alpha$ over $\QQ$.

In the definition of a reduction of $\mathfrak a$ modulo a \emph{well-chosen} prime $p$, our aim is preserving some features of $\mathfrak a$.

We fix a set of generators $(f_1,\dots,f_s)$ in $\QQ(\alpha)[\mathbf X]$.
We multiply each $f_i$ with a scalar $c_i$ such that $c_i\cdot f_i\in\ZZ[\alpha][\mathbf X]$ and the coefficients of $c_i\cdot f_i$ have g.c.d. (on the integers) equal to 1: we call such a set of generators in $\ZZ[\alpha][\mathbf X]$ \emph{primitive}. We keep on writing $f_i$ for $c_i\cdot f_i$.

We now consider a prime integer $p$ such that $q(T)$ splits in $\ZZ/p\ZZ=\FF_p$ in the following way: 
\begin{equation}\label{buonp}
q(T)=S_1(T)\cdot S_2(T)\,\mathrm{mod}\, p, \; \deg S_1(T)=1,\deg S_2(T)=s-1, \; \gcd(S_1(T),S_2(T))=1.\end{equation}

Thanks to Chebotarev's Density Theorem (\cite{LCheb}), we know that there are infinite prime integers $p$ for which (\ref{buonp}) holds.

Let $\beta_p$ be the only root of $S_1(T)$ in $\ZZ/p\ZZ$, $0\leq \beta_p\leq p-1$: $S_1(\beta_p)=0$. We then define the following map, from the ring $\ZZ[\alpha]$ of $\ZZ$ to the finite field $\FF_p=\ZZ/p\ZZ$:
\begin{equation}\label{hommodp}
\begin{split}
\psi_p: \ZZ[\alpha]& \rightarrow \FF_p \\
\alpha  & \mapsto \beta_p \\
a\in \ZZ & \mapsto a \mod p.
\end{split}
\end{equation}

This definition on the generators obviously extend to a homomorphism of rings, well-defined because of the choice of $p$ and consequently of $\beta_p$.

We can then extend this homomorphism to the polynomials:
\begin{equation*}
\begin{split}
\psi_p: \ZZ[\alpha][\mathbf X]& \rightarrow R_p=\FF_p[\mathbf X] \\
f=\sum_I a_I	\mathbf X^I&\mapsto \tilde f=\sum_I\psi_p(a_I)\mathbf X^I.
\end{split}
\end{equation*}

If we consider $\mathfrak a=(f_1,\dots,f_s)\subseteq\QQ(\alpha)[\mathbf X]$, we can assume that the chosen generators are primitive  and are in $\ZZ[\alpha][\mathbf X]$; we define $\tilde{\mathfrak a}=(\tilde f_1,\dots,\tilde f_s)\subseteq R_p$.

\begin{remark}
 Observe that the definition of $\tilde{\mathfrak a}$ is independent on the chosen set of generators of $\mathfrak a$:\\ 
 if $\mathfrak a=(f_1,\dots,f_s)=(f_1',\dots,f_l')$  then $(\tilde f_1,\dots,\tilde f_s)=(\tilde f_1',\dots,\tilde f_l')$
as ideals in $R_p$.
\end{remark}

\begin{example}
Consider the ideal
\[
 \mathfrak a=(3\,{Y}^{2}-2\,\sqrt {3}ZX,3\,YX-\sqrt {3}\sqrt {2}Z,2\,{X}^{2}-\sqrt {2}Y)\subseteq \CC[X,Y,Z].
\]
 This set of generators has coefficients in the algebraic extension $\QQ(\sqrt{2}+\sqrt{3})$, which is normal, $[\QQ(\sqrt{2}+\sqrt{3}):\QQ]=4$, the minimal polynomial of $\sqrt{2}+\sqrt{3}$ is $q(T)=T^4-10T^2+1$.

Consider now $p=23$:
\[
 q(T)=(T+21)\cdot(T+ 12)\cdot(T+ 2)\cdot(T+ 11) \mod p.
\]
We consider the homomorphism $\psi_p$ such that $\psi_p(\sqrt{2}+\sqrt{3})=21$.
With this definition of $\psi_p$, we have that: $\psi_p(\sqrt{2}\sqrt{3})=11, \psi_p(\sqrt{2})= 5, \psi_p(\sqrt{3})=16$.\\
So $\tilde{\mathfrak a}=(3\,{Y}^{2}+14\,ZX,3\,YX+12\,Z,2\,{X}^{2}+18\,Y)$.
\end{example}

 We can choose  a prime $p$ satisfying (\ref{buonp}) using the following lemma.

\begin{lemma}\label{rootinQp}[\cite{BCG}, Lemma 12]
Let $q(T) \in \ZZ[T]$ be a polynomial and $p$ a prime number such that $p$ divides $q(0)$, $p$ does not divide the discriminant of $q(T)$ and $p>\deg(q(T))$.\\
Then there exists a root in $\QQ_p$ of $q(T)$, considered as a polynomial in $\QQ_p[T]$.
\end{lemma}

We would like to understand in which cases computations modulo a prime integer $p$  preserve  the affine Hilbert function of the ideal $\mathfrak a$.
 From Proposition \ref{propositionHAff}, we  can bring back our problem about the choice of a prime $p$ preserving  the affine Hilbert function  to the choice of a ``good'' $p$ preserving the initial ideal (w.r.t. some term ordering) of $\mathfrak a$.

Assume that $(g_1,\dots, g_s)$ is a Groebner Basis of $\mathfrak a$ w.r.t. $\preceq$, degree compatible term ordering,  and that these polynomials are primitive; we now choose a prime $p$  satisfying (\ref{buonp}) for the minimal polynomial of $\alpha$. Then, we define $\tilde{\mathfrak a}:=(\tilde g_1,\dots,\tilde g_s)\subseteq R_p$.



Proposition \ref{propositionHAff} gives  a necessary condition for a prime integer $p$ to preserve the affine Hilbert function of $\mathfrak a$: if computations modulo $p$ preserve $HF^a_{R/\mathfrak a}$ then they also preserve $LM_\preceq(\mathfrak a)$ with respect to a degree-compatible term ordering $\preceq$.

 We will show that a finite number of primes $p$ does not satisfy this necessary condition.

\begin{lemma}
 Consider $\alpha$ algebraic number on $\QQ$, $[\QQ(\alpha):\QQ]=s$, $L_1, \dots, L_N$ non-zero elements of $\ZZ[\alpha]$,  $L_i=\sum_{j=0}^{s-1} a_j^{(i)}\alpha^j $. There is a finite number of prime integers $p$ such that $\psi_p(L_i)=0$ for some $i$.
\end{lemma}

\begin{proof}
 We will proceed by contradiction.

Suppose that there are infinite prime integers $p$ such that $\psi_p$ maps to zero at least one of the $L_i$'s.
In particular there is an index $\tilde {\imath}$ such that $L_{\tilde {\imath}}$ is mapped to zero by infinite maps $\psi_p$. 
We can define the polynomial $\tilde L(T)=\sum_{j=0}^{s-1} a_j^{(\tilde {\imath})} T^j$. 

If $\deg \tilde L(T)=0$, $\tilde L(T)=L_{\tilde {\imath}}$ is an integer and  there is only a finite number of $\psi_p$ mapping  $L_{\tilde {\imath}}$ to zero, because there is only a finite number of $p$'s dividing it. 

We can then assume that $d=\deg \tilde L(T)\geq 1$.  We consider a prime $p$ such that 
\[
\psi_p(L_{\tilde{\imath}})=0, \quad p\geq \Vert \tilde L(T)\Vert^s\Vert q(T)\Vert^{d},
\]
 where $q(T)$ is the  minimal polynomial of $\alpha$; we can choose such a $p$ since the set of $p$'s we are looking at is supposed to be infinite.\\
Observe that $\psi_p(L_{\tilde{\imath}})=\tilde L(\beta_p) \mod p=0$. This means that both $\tilde L(T)$ and $q(T)$ can be divided by $(T-\beta_p)$ modulo $p$. But since we chose  $p\geq \Vert \tilde L(T)\Vert^s\Vert q(T)\Vert^{d}$,  we can apply  \cite{vzg}, Lemma 16.20: $\deg\left(\gcd(\tilde L(T), q(T))\right)\geq 1$.\\ Since $\deg \tilde L(T)<\deg q(T)$, this contradicts the fact that $q(T)$ is irreducible.
\end{proof}

\begin{lemma}\label{GBmodp}
 Let $\mathfrak a$ be an ideal in $\QQ(\alpha)[\mathbf X]$, $\preceq$ a  term ordering and $G=\{g_1,\dots,g_r\}$ 
 a Groebner Basis of $\mathfrak a$ with respect to $\preceq$, $p$ a prime integer satisfying (\ref{buonp}) for the minimal polynomial of $\alpha$.\\
If $\psi_p$ does not map to 0 any of the coefficients of the leading monomials of the polynomials in $G$, then $\tilde G=\{\tilde g_1,\dots,\tilde g_r\}$ is a Groebner Basis of $\tilde{\mathfrak a}$ with respect to $\preceq$.
\end{lemma}

\begin{proof}
 Since $G$ is a system of generators for the ideal $\mathfrak a$, then $\tilde G$ is a system of generators for $\tilde{\mathfrak a}$.\\
Consider $g\in \mathfrak a$ and its representation with respect to the basis $G$ (eventually multiplying for $c \in \ZZ$, to eliminate the denominators):
\[
 g=\sum_{i=1}^ra_ig_i, \quad a_i \in R.
\]
Then we have the corresponding representation of $\tilde g$ with respect to the basis $\tilde G$:
\[
 \tilde g=\sum_{i=1}^r\tilde a_i\tilde g_i.
\]
We have that $LM(\tilde g)=\max_{i=1,\dots,r}\{LM(\tilde a_i \tilde g_i)\}$, where the ``$\max$'' is taken with respect to $\preceq$. 

Since $LM(\tilde a_i\tilde g_i)=LM(\tilde a_i)LM(\tilde g_i)$ for every $i$ and  $\psi_p$ does not map to zero the coefficients of the leading monomials of the polynomials in the Groebner basis, we immediately have $LM(\tilde g) \in (LM(\tilde g_1),\dots,LM(\tilde g_s))$ and so $\tilde G$ is a Groebner Basis for $\tilde{\mathfrak a}$.
\end{proof}

\begin{thm}\label{finitebadphilb}
 Let $\mathfrak a$ be an ideal in $R=\QQ(\alpha)[\mathbf X]$. Then for a finite number of prime integers $p$, we have that
\[
 HF^a_{R/\mathfrak a}(i)\neq HF^a_{R_p/\tilde{\mathfrak a}}(i)\quad \text{for some }i.
\]
\end{thm}

\begin{proof}
We fix a degree compatible term ordering $\preceq$ and consider the Groebner Basis $G=\{g_1,\dots,,g_r\}$ of $\mathfrak a$. Thanks to Proposition \ref{propositionHAff}, $HF^a_{R/\mathfrak a}(i)\neq HF^a_{R_p/\tilde{\mathfrak a}}(i)$ for some $i$ only if the initial ideal of $\mathfrak a$  differs from $\tilde{\mathfrak a}$. \\
We apply  Lemma \ref{GBmodp}: there is only a finite number of primes $p$ such that the initial ideals of $\mathfrak a$ and $\tilde{\mathfrak a}$ are different.
\end{proof}

\begin{corollary}\label{Hilbdimp}
 There is a finite number of prime integers $p$ such that the affine Hilbert function and dimension of $\mathfrak a$ and $\tilde{\mathfrak a}$ differ. If $\mathfrak a$ is a primary component of some ideal, the same holds for the multiplicity of $\mathfrak a$ and $\tilde{\mathfrak a}$.
\end{corollary}


\section{An exact strategy: Elimination of Variables and Colon Ideals}\label{exactstrat}

In this section, we will present an exact technique, which combines elimination of variables, absolute fac\-to\-ri\-za\-tion and computation of colon ideals and gives the primary decomposition of an ideal.

  This technique is intuitive and immediate from the geometric point of view: the elimination of variables geometrically corresponds to projection on some linear space.\\
   The use of projections reduces the problem of decomposing an algebraic set to a multivariate factorization; this was first showed at the beginning of the XX-th century (see \cite{herm}). After that, Seidenberg in \cite{seid} used a more rigorous formalism then Hermann's to establish which ideal operation can be actually computed, depending also on the features of the polynomial ring we are working on.\\
   The computation of a colon ideal geometrically corresponds to take off the points of an algebraic set from another one. This relation between the ``difference'' of varieties and colon ideals is well-known too (see \cite{cox}, Chapter 4, \S 4).
   
    The strategy of elimination of variables and computation of quotient ideals is also used in \cite{cct}. The authors in \cite{cct} can move around the computational effort of using \emph{generic} projections (in the sense of Definition \ref{1to1}) by using another powerful tool, which is the relation between flatness and variation of staircases. Thanks to the study of the flatness of the variation of the staircases, they can split the ideal according to the splitting of the projection, even if it is not a generic one; they repeat the process for the ideals obtained by splitting and at each step they have the dimension or the multiplicity of the ideals decreasing, so their algorithm terminates. The algorithm in \cite{cct} has different variants and can give the strict, isoradical and reduced  equidimensional decomposition of a polynomial ideal in $\QQ[\mathbf X]$ (see \cite{cct}, Introduction, for the different kind of decompositions).  \\
    Unfortunately, even if in \cite{cct} the proofs are given for a field $k$, with suitable properties that $\overline Q$ has, the  implementation of the algorithm  in \cocoa deals only with  $\QQ$ or a finite field as coefficient fields, so this algorithm is not used in a CAS for the computation of an absolute primary decomposition (see Section \ref{tricks}).
    
   
   \smallskip
   
   Although the basic idea of the technique we are going to use is well-known, for lack of a complete and accessible reference on the whole strategy relating the algebraic and geometrical point of view , we will present it in an exhaustive way.

The decomposition algorithm that we can obtain  using these techniques (Algorithm \ref{algdec}) is exact but not useful in practice: the computations needed are quite long and hard to perform. Anyway, we will investigate this method in details since later (Section \ref{compmodp}) we will modify this strategy giving up the exact computations in order to gain velocity in computations, but preserving some information about the irreducible and reduced components, namely their affine Hilbert Function.

\medskip

First of all, we now investigate some properties of the projection of varieties.
For the omitted proofs, see \cite{svw}.



\begin{definition}
A \emph{linear projection} is a surjective affine map:
\begin{equation}
\begin{split}
\pi:\CC^n&\rightarrow \CC^m\\
\xx=(x_1,\dots,x_n)\in \CC^n&\mapsto (L_1(\xx),\dots,L_m(\xx)) \quad \text{ with } L_i(\mathbf X)=a_{i_0}+\sum_{j=1}^na_{i_j}X_j.
\end{split}
\end{equation}
In a similar way, we can define a linear projection to projective space, considering
\begin{equation}
\begin{split}
\pi_L:\PP^n\setminus L&\rightarrow \PP^m\\
\xx=[x_0:x_1:\dots:x_n]&\mapsto [L_0(\xx):L_1(\xx):\dots:L_m(\xx)], \quad L_i(\mathbf X)=\sum_{j=0}^na_{i_j}X_j.
\end{split}
\end{equation}
$L\subseteq \PP^n$ is the point of intersection of the linear equations $L_i$ and it is the \emph{center} of the projection.\\
Two projections $\pi_1$ and $\pi_2$ from $\PP^n$ to $\PP^m$ are \emph{equivalent} if they have the same center $L$.   
\end{definition}

We now consider a projection $\pi:\CC^n\rightarrow\CC^m$ and we restrict it to an equidimensional algebraic set  $\C$ of $\CC^n$ with dimension $c$. It is not always true that $\pi_\C$ is proper: for instance, the projection of the hyperbola defined by the ideal $(XY-1)$ on the $X$-axis  is the $Y$-axis without the origin; with this particular choice of the linear space to project on, the restriction of $\pi$ to the hyperbola is not proper.

\begin{proposition}[\cite{svw}, Lemma 5.1]\label{1to1}
Let $\C$ be a closed algebraic set of $\CC^n$ all of whose irreducible components are of dimension $c$.   For a general linear projection $\pi:\CC^n\rightarrow \CC^m$, $m\geq c+1$, the map $\pi_\C$ is proper and generically one-to-one.\\
In particular $\pi(\C)$ is a closed algebraic set of $\CC^m$ of degree equal to the degree of $\C$.
\end{proposition}

\begin{corollary} \label{corproiez} Let $\C$ be an equidimensional algebraic set of dimension $c$ in $\CC^n$. For a general projection $\pi:\CC^n\rightarrow \CC^{c+1}$, the following hold:
\begin{enumerate}
\item if $\C_i$ and $\C_j$ are two distinct irreducible components of $\C$, then $\pi(\mathcal \C_i)\neq \pi(\mathcal \C_j)$;
\item  $\C$ is irreducible if and only if the polynomial defining $\pi(\C)$ is absolutely irreducible;
\item if $D(T_1,\dots,T_c)$ is the polynomial defining the projection $\pi(\C)$ and we consider its absolute factorization $D=D_{1}^{m_1}\cdots D_{s}^{m_s}$, then $s$ is exactly the number of distinct irreducible components of $\C$, $m_i$ is the multiplicity of the component $\C_i$ and $\deg D_{i}$ its degree.
\end{enumerate}
\end{corollary}

We now assume that $\C$ is an equidimensional algebraic set of dimension $c$ defined by the ideal $\mathfrak a\subseteq \CC[\mathbf X]$, $\mathfrak a=(F_1,\dots,F_r), F_i\in \QQ[\mathbf X]$; we further assume that we performed a generic linear change of coordinates with integer coefficients; in this way we can consider the projections on linear spaces defined by equations of kind $X_i=0$   to be generic.

We consider the projection $\pi_1:\CC^n\rightarrow H_1$, with $H_1$ the linear space of dimension $c+1$ defined by the equation $X_{c+2}=\dots=X_{n}=0$. We call $\mathcal Z_1$ the projection of the algebraic set $\C$ on $H_1$. This is a hypersurface in $H_1$ and its decomposition is equivalent to the absolute factorization of the multivariate polynomial $D_1(X_1,\dots,X_{c+1})\in \QQ[\mathbf X]$ defining $\mathcal Z_1$ on the linear space $H_1$. Furthermore, the components of $\mathcal Z_1$ are in one-to-one correspondence with the irreducible components of $\C$. 

If $D_1=D_{11}^{m_1}\cdots D_{1s}^{m_s}, D_{1j}^{m_j}\in\CC[\mathbf X]$, is the absolute factorization of $D_1$, each factor $D_{1j}$ defines in $\CC^n$ a ruled surface, a cylinder, containing the component $\C_j$. \\
Algebraically, thanks to the generic change of coordinates and to Corollary \ref{corproiez}, if $\mathfrak a=\bigcap\mathfrak q_j$, then there is only one absolute factor $D_{1j}$ in the absolute factorization of $D_1$ such that $D_{1j}^{m_j}\in \mathfrak q_j$. Furthermore, for $m<m_j$, $D_{1j}^m\notin \mathfrak q_j$.

We can then start considering the ideal $\mathfrak a+(D_{1j}^{m_j})$. But $V(\mathfrak a+(D_{1j}^{m_j}))$ contains not only the component $\C_j$, but also the points of the sets $V(D_{1j}^{m_j})\cap \C_k$ for $k\neq j$.

Thinking of the corresponding varieties $V(\mathfrak q_i)$ and $V(D_{1j}^{m_j})$, we can classify the primary components of the ideal $\mathfrak a+(D_{1j}^{m_j})$ as follows:
\begin{itemize}
\item a primary component of dimension $c$, which is the primary component $\mathfrak q_i$ such that $D_j^{m_j}\in \mathfrak q_i$ (and in this case we write $i=j$). Obviously in this case $\mathfrak q_j+(D_{1j}^{m_j})=\mathfrak q_j$ and $\mathfrak p_j+(D_{1j})=\mathfrak p_j$;
\item  other primary components of dimension $<c$ when $V(\mathfrak q_i)\cap V(D_{1j}^{m_j})$ is non-empty and it does not contain the component $\C_j$.
\end{itemize}

Finally, the ideal $\mathfrak a+(D_{1j}^{m_j})$ defines the component $\C_j$ with some extra components of dimension $<c$.
In order to avoid these extra components, we can repeat the same steps with $(n-c-1)$ more generic projections $\pi_k$ from $\CC^n$ onto the linear spaces $H_k$, defined by equations of kind $X_{i}=0$, thanks to the chosen generic coordinates (see also Section \ref{exactalgs} for more details). If we consider $\pi_k(\C)=\mathcal Z_k$, the polynomial defining $\mathcal Z_k$ on $H_k$ is again a rational multivariate polynomial   $D_k$ in $c+1$ variables. We compute its absolute factorization and we obtain another cylinder containing $\C_j$, defined by the factor $D_{kj}$.

Actually, in the absolute factorizations of the polynomials $D_k$  there may be several factors with the same degrees and multiplicity. For instance, this  happens when one of the components $\C_j$ is a non-pure rational one (see Definition \ref{nonrat}). This can give ambiguity in matching the factors of the polynomials $D_k$ whose zero set contains the component $\C_j$.

 In order to match the factors defining the cylinders containing the same component, we can look at the Hilbert dimension of the ideal $\mathfrak a+(\sum_{k=1}^{c+1}(D_{kj}^{m_j}))$. This dimension is $c$ if and only if all of the sets $V(D_{kj}^{m_j})$ contain the same irreducible component of $\C$; this follows from Corollary \ref{corproiez}, 1.

In order to find the correct matchings, we will not compute the Hilbert dimensions of the ideals $\mathfrak a+(\sum_{k=1}^{c+1}(D_{kj}^{m_j}))$ for every possible $(n-c)$-uple of factors; we will compute the Hilbert dimension only for the matchings such that all the factors $D_{kj}$ have the same degree and multiplicity. Furthermore, we get an almost certain probabilistic check by considering a generic section with a linear space of dimension $n-c$: that is, we can look for the matchings such that the ideal $\mathfrak a+(\sum_{k=1}^{c+1}(D_{kj}^{m_j}))$ with $(n-c)$ variables evaluated in some integer values  is zero-dimensional and nonempty (see Algorithm \ref{algmatch}).

Once matched the absolute factors found through the projections, $\{D_{kj}\}_{j=1,\dots,c}$ (after re-indexing of the factors), the ideal
\begin{equation}\label{dec1}
\mathfrak a_j:=\mathfrak a+(\sum_{k=1}^{c+1}(D_{kj}^{m_j}))
\end{equation}
 is ``almost'' the primary ideal corresponding to the component $\C_j$: there are some embedded components left, which geometrically  are in $\C_j\cap  \C_l$, for $l\neq j$.

\begin{lemma}\label{puntproj}
Consider $(n-c)$ projections $\pi_i$ from $\CC^n$ to $\CC^{c+1}$,  pairwise not equivalent.
If there are points  $P_1, P_2\in \CC^n$ such that
\[
\pi_i(P_1)=\pi_i(P_2)\quad \forall i=1,\dots, n-c,
\]
then $P_1=P_2$.
\end{lemma}

\begin{proof}
We write $\pi_k(P)=(L_{1}^{(k)}(P),\dots, L_{c+1}^{(k)}(P))$ with  $L_i^{(k)}(\mathbf X)=a_{i0}^{(k)}+\sum_{j=1}^n a_{ij}^{(k)}X_j$. Remark that since a projection is a surjective map, $L_i^{(k)}\neq L_j^{(k)}$, $1\leq i< j\leq c+1$; furthermore, we will be interested in the matrix whose lines are the vectors $\left(a_{ij}^{(k)}\right)_{j=1,\dots,n}$ for $k=1,\dots n-c,i=1,\dots,c+1$. It is a matrix with $(n-c)(c+1)$ lines and $n$ columns; since the chosen projections are pairwise non-equivalent, this matrix has maximal rank: $(n-c)(c+1)\geq n$ for all $0\leq c\leq n-1$, so the considered matrix has rank $n$.

$P_1$ and $P_2$ have the same image under $\pi_j$ if and only if $L_i^{(j)}(P_1)=L_i^{(j)}(P_2)$, $i=1,\dots,c+1$. 
We obtain $(n-c)\cdot(c+1)$ equations of kind 
\[
 \sum_{j=1}^na_{ij}^{(k)}(x_j^{(1)}-x_j^{(2)})=0, \quad P_1=(x_j^{(1)})_{j=1,\dots,n},P_2=(x_j^{(2)})_{j=1,\dots,n}.
\]
The unique solution to this system of equations  is the trivial one, since the matrix associated to the linear system is exactly $\left(a_{ij}^{(k)}\right)$ whose rank is $n$, and so we have that $P_1=P_2$.
\end{proof}

\begin{corollary}
Consider $\C$ equidimensional algebraic set of dimension $c$ in $\CC^n$ and  $\C_1$ and $\C_2$ irreducible components of $\C$.
If $P_1\in \C_1$ and $P_2\in \C_2$ are such that 
\[
\pi_i(P_1)=\pi_i(P_2),
\]
for $(n-c)$ generic projections $\pi_i:\CC^n\rightarrow \CC^{c+1}$, then $P_1=P_2$ is a point in $\C_1\cap \C_2$.
\end{corollary}


Thanks to Lemma \ref{puntproj} and the one-to-one correspondence between the components of $\C$ and the components of $\pi(\C)$, for a general projection $\pi$, if we consider the ideal $\mathfrak a_j$ as in (\ref{dec1}), $V(\mathfrak a_j)$ is  $\C_j$; however the polynomials in $\mathfrak a_j$ vanish on the points of $\C_k\cap \C_l$, $l\neq j$, with multiplicity higher than $\mathrm{mult} (\C_j)$.

So, for the moment, we have an ideal such that its zero set contains the irreducible component $\C_j$ but there are some embedded components; we now show that the points of these embedded components are also contained in the zero set of the singular locus of $\mathfrak a_j$ (Definition \ref{defsingloc}). \\
Considering again  primary decompositions, for what concerns $\mathfrak a_j$ we have that:
\begin{equation}
 \mathfrak a_j=\mathfrak a+(\sum_{k=1}^{c+1}(D_{kj}^{m_j}))=\mathfrak q_j \cap (\bigcap_{i=1}^v \mathfrak b_i)
\end{equation}
where $\mathfrak b_i$ is a primary ideal of dimension $<c$. Geometrically, the primary components $\mathfrak b_i$ correspond to the irreducible component of $\C_j\cap \C_k$, $j\neq k$. These components are in the \emph{singular locus} of $R/\mathfrak a$.

Here we just recall the algebraic definition of singular locus and the useful Jacobian criterion.

\begin{definition}\label{defsingloc}
 Let $\mathfrak a$ be an ideal in $R=\KK[\mathbf X]$, $\KK$ perfect field, 
$\mathfrak a=(f_1,\dots,f_s)$. A prime ideal $\mathfrak p$ containing $\mathfrak a$ is in the \emph{singular locus} of $R/\mathfrak a$ if the localization of $R/\mathfrak a$ at $\mathfrak p$ is not a regular local ring.
\end{definition}

With an  abuse of notation, we will say ``singular locus of $\mathfrak a$'' for the singular locus of $R/\mathfrak a$.

\begin{proposition}[\cite{Eis}, Corollary 16.20]
 Let $\mathfrak a$ be an ideal in $\KK[\mathbf X]$, $\KK$ perfect field, $\mathfrak a$ is equidimensional with dimension $c$, $\mathfrak a=(f_1,\dots,f_s)$. Let $J$ be the ideal generated by the $(n-c)\times (n-c)$-minors of the Jacobian matrix $\left({\partial f_i}/{\partial X_j}\right)$. Then $J$ defines the singular locus of $\mathfrak a$: a prime $\mathfrak p$ contains $J$ if and only if $\mathfrak p$ is in the singular locus of $\mathfrak a$.
\end{proposition}

We can then compute easily the equations defining the singular locus of $\mathfrak a$:
\begin{itemize}
\item Compute the jacobian matrix of  $\mathfrak a$;
\item Compute $\{M_l\}$, the $(n-c)\times (n-c)$ minors of the jacobian matrix;
\item The singular locus of $\mathfrak a$ is defined by $(M_l)$.
\end{itemize}

We are then interested in considering $\mathfrak a_j$ and removing the ``embedded'' primary components which contains the singular locus of $\mathfrak a$. We can do this for the ideals $\mathfrak a_j$ containing a reduced component, that is the ones obtained from irreducible factors of multiplicity 1. 

In general, for $\mathfrak a_1,\mathfrak a_2$ ideals in a ring $R$:
$$(\mathfrak a_1:\mathfrak a_2)=\{f \in R|f\cdot \mathfrak a_2\subseteq \mathfrak a_1\}$$
If we consider an ideal generated by a single element $f$ of $R$, we will simply write $(\mathfrak a:f)$ instead of $(\mathfrak a:(f))$.

\begin{lemma}[\cite{atiyah}, Lemma 4.4.]\label{quotprim}
 Let $\mathfrak q$ be a $\mathfrak p$-primary ideal, $f$ an element of $R$. Then 
 \begin{enumerate}
\item if $f \in \mathfrak q$ then $(\mathfrak q:f) = (1)$; 
\item if $f\notin \mathfrak q$ then $(\mathfrak q:f)$ is $\mathfrak p$-primary, and therefore $\sqrt{(\mathfrak q:f)} = \mathfrak p$; 
\item if $f\notin \mathfrak p$ then $(\mathfrak q:f)= \mathfrak q$. 
\end{enumerate}
\end{lemma}

\begin{proposition}\label{pulitura}
 In the previous setting, if $D_{ij}$ are factors of multiplicity one in the factorizations of the polynomials $D_i$, such $D_{ij}\in \mathfrak q_j$, $\mathfrak a_j=\mathfrak a+(\sum_{k=1}^{c+1}(D_{kj}))$  and $M$ is a $(n-c)\times(n-c)$ minor of the Jacobian matrix of $\mathfrak a$, then  $(\mathfrak a_j:M)$ is exactly $\mathfrak p_j$, the prime ideal defining the irreducible component $\C_j$ of $\C=V(\mathfrak a)$.
\end{proposition}
\begin{proof}
Consider the primary decomposition of $\mathfrak a_j$:
\[
  \mathfrak a_j=\mathfrak a+(\sum_{k=1}^{c+1}(D_{kj}))=\mathfrak q_j \cap (\bigcap_{i=1}^v \mathfrak b_i)
\]
where $\mathfrak q_j$ is the only component of dimension $c$, while $\mathfrak b_i$ are embedded components of dimension $<c$. These embedded components correspond to the intersection of $\C_j$ with the other irreducible components of $\mathfrak a$. Furthermore, since the factors $D_{kj}$ have multiplicity one in the factorizations of the $D_i$, this means (Corollary \ref{corproiez}) that $\sqrt{\mathfrak q_j}=\mathfrak q_j$, that is $\mathfrak q_j$ is a prime component, so we can write $\mathfrak p_j$ for $\mathfrak q_j$.

This means that if $M$ is an equation of the singular locus of $\mathfrak a$, then $M\in \mathfrak b_i$ for $i=1,\dots,v$ and so $(\mathfrak b_i:M)=(1)$. On the other hand, $\mathfrak p_j$ is the only primary (and prime) component of $\mathfrak a_j$ which does not contain the singular locus: so $(\mathfrak p_j:M)=\mathfrak p_j$.

So
\[
\mathfrak a_j:M=(\mathfrak p_j:M)\cap\left(\bigcap_{i=1}^v (\mathfrak b_i:M)\right)=\mathfrak p_j.
\]
\end{proof}

\begin{remark}
Proposition \ref{pulitura} applies only for components of multiplicity 1 (that is, for factors in the absolute factorization of multiplicity 1).\\
In fact, if we consider an ideal $\mathfrak a_j=\mathfrak a+(\sum_{k=1}^n(D_{kj}^{m_j})), m_j\geq 2$, we have that this ideal contains the singular locus of $\mathfrak a$, so $M\in \mathfrak a_j$ and (using Lemma \ref{quotprim}), $(\mathfrak a_j:M)=(1)$.

For components of multiplicity $>1$, we may ``clean up'' at least some of the embedded components of the ideal $\mathfrak a_j$ in the following way:
\begin{itemize}
 \item we use the described strategy for the components of multiplicity one, obtaining the set of prime ideals $\{\mathfrak p_{i_1},\dots,\mathfrak p_{i_v}\}$ which are the ideals for the irreducible and reduced components $\{\C_{i_1},\dots, \C_{i_v}\}$
\item since we cannot use the colon ideal with respect to a generator of the singular locus, we then compute
\[
\mathfrak b:=(\cdots((\mathfrak a_j: f_{i_1}):f_{i_2})\cdots :f_{i_v}),
\] 
for $f_{i_j}\in \mathfrak p_{i_j}$.
\end{itemize}
In this way, we can  ``clean'' part of the points of the embedded components. For the remaining ones, for any component of multiplicity $>1$ we should consider the  absolute factors $D_{1j}$ of $D_1$ not contained in this primary component, and compute again some ``nested'' colon ideals, starting from $\mathfrak b$, with respect to the  $D_{1j}$. 
\end{remark}

\subsection{Algorithms}\label{exactalgs}

In this section we will summarize the strategy  of elimination of variables and colon ideals, writing down the main algorithm for primary decomposition (Algorithm \ref{algdec}) and an auxiliary one to match $(n-c)$-uples $D_{1j}^{m_j},\dots, D_{cj}^{m_j}$ (Algorithm \ref{algmatch}).

\smallskip

In order to consider $(n-c)$ projections $\pi_i:\CC^n\rightarrow \CC^{c+1}$ (pairwise non-equivalent) and compute the polynomials $D_i$ which is zero on $\pi_i(\C)$, at the beginning of Algorithm \ref{algdec} we perform a generic change of affine coordinates and then use projections on \lq\lq coordinate\rq \rq  \,affine linear spaces, namely linear spaces defined by equations of kind $X_j=0$. 
While in the previous section we used \lq\lq$H$\rq\rq \ for the coordinate affine linear space
, we will now use \lq\lq$H$\rq\rq \ to denote the set of indexes  of the variables $X_j$ to eliminate to obtain the projection of the curve:
\begin{equation}\label{spazi}
X_j=0, \quad j\in H=\{j_1,\dots, j_{n-c-1}\}\subseteq{1,\dots, n}=[n].
\end{equation}
We can always find $(n-c)$  sets of indexes $H_i$, $i=1,\dots, n-c$, such that $H_l\neq H_m$ for every $l\neq m$.\\
If $H=\{j_1,\dots,j_{n-c-1}\}$, we write  
$\{X_l\}_{l\in [n]\setminus H}$ for $\{X_{l_1},\dots, X_{l_{c+1}}\}$, with $[n]\setminus H=\{l_1,\dots,l_{c+1}\}$.

\smallskip

The actual computation of the polynomial $D_i$ is obtained by the computation of a Groebner Basis with respect to an elimination term ordering.\\
If we consider the set of indexes $H$ as in (\ref{spazi}), we can compute a Groebner Basis which eliminates from $\mathfrak a$ the variables $X_j$, $j\in H$; from now on we will denote such a Groebner Basis with $GB_{elim(H)}(\mathfrak a)$, or, if we  fix $H_i$, $i=1,\dots, n-c$, we simply write  $GB_{elim(i)}(\mathfrak a)$. Then $D_i$ is the only generator of $GB_{elim(i)}(\mathfrak a)$ such that
\begin{equation}\label{polyproj}
D_i \in GB_{elim(i)}(\mathfrak a)\cap \QQ[X_l]_{l \in [n]\setminus H_i}
\end{equation}

We will always assume that $H_1=\{c+2,\dots,n\}$.

\smallskip

In Algorithm \ref{algmatch}, in order to match the absolute factors corresponding to the same component, we use the fact that if an algebraic set has dimension $<c$, then its general section with a linear space of dimension $(n-c)$ is empty; this means that if  $\mathfrak b$ is an ideal of dimension $<c$, then the affine Hilbert Dimension  of $\mathfrak b$ with $(n-c)$ variables evaluated in integer values is generically $-1$. Furthermore, thanks to the generic change of coordinates, we can evaluate the variables in 0.

Thanks to the results of Section \ref{defgen}, we can modify Algorithm \ref{algdec}  and \ref{algmatch}   and avoid repeating the computations for conjugate primary components. We present the Algorithms without this for simplicity and also because they will not actually be used in this form.

Indeed, Algorithm \ref{algdec} is exact, but in practice the computations are too heavy: the main computational difficulties are  at Step \ref{resultant} (because of the computation of the Groebner Basis), Step \ref{factass} (because of the multivariate absolute factorization) and Step \ref{sat} (because of the computation of the colon ideal)   of Algorithm \ref{algdec}.

In the next section we  show that we can perform all these computations modulo a well-chosen prime integer $p$, obtaining Algorithm \ref{algdecmodp}. Its output is no more exact, in the sense that it does not return the ideals of the reduced components, but it returns information about the ideals of the reduced components, in particular their initial ideal and affine Hilbert functions.

This will be possible observing that the computations of the troublesome steps of Algorithm \ref{algdec} are actually obtained through Groebner Basis and  that a good choice of $p$  preserves along the computations the initial monomials of the polynomials in the Groebner Basis and so the Hilbert functions of the ideals.


\newpage

\begin{algorithm}[htbp]
\caption{Exact Decomposition of an equidimensional algebraic set}    
\label{algdec}                           
\KwData {$\mathfrak a\in \QQ[X_1,\dots,X_n]$, $\mathfrak a$ of pure dimension $c$}
\KwResult {The number of irreducible components, their degrees and multiplicities.\\
If the multiplicity of the component $\C_j$ is 1, a system of generators of the ideal $I(\C_j)$; if the multiplicity of the component is $\geq 2$, the equations of $(n-c)$ hypersurfaces ``isolating'' the component from the other ones.}
\PREPROCESSING Perform a generic affine change of coordinates on the generators of $\mathfrak a$\;
 Compute $(n-c)$ polynomials, eliminating from $\mathfrak a$ the set of variables corresponding to the set of indexes $H_i$, $i=1,\dots, n-c$: $D_i$ is the only generator of $GB_{elim(i)}(\mathfrak a)$ with the property (\ref{polyproj}), $i=1,\dots,n-c$ \;\nllabel{resultant}
 Perform the absolute factorization of $D_i$, $i=1,\dots,n-c$ 
$
D_i=D_{i1}^{m_1}\cdots D_{is}^{m_{s}}, \quad i=1,\dots,n-c
$
\;\nllabel{factass}
Match the $D_{ij}$'s  through Algorithm \ref{algmatch} in such a way that (after re-numbering of the factors) $\mathfrak a_{j}=\mathfrak a+(\sum_{k=1}^{n-c}(D_{kj}))$ contains the component $\C_j$\;
 Compute the Jacobian matrix of $\mathfrak a$ and  a minor of size $(n-c) \times (n-c)$, $M$ \;\nllabel{jacob}
\For{$i$ from $1$ to $s$}{
\If{$m_i$ is 1}{
 Compute $\mathfrak a_{j}^{(M)}$ the quotient ideal of $\mathfrak a_{j}$ with $M$: $\mathfrak a_{j}^{(M)}:=\mathfrak a_{j}:(M)$\;\nllabel{sat}
}
}
\RETURN \ \\
for $j=1,\dots,s$:  $\deg D_{1j}$ degree of the component, $m_j$ multiplicity of the component;\\
 if $m_j\geq 2$, $\mathfrak a_{j}$ ideal isolating the component;\\
  if $m_j=1$,  $\mathfrak a_{ij}^{(M)}=I(\C_j)$.
\end{algorithm}

\begin{algorithm}[htbp]        
\caption{Matching of factors through Hilbert Dimension}          
\label{algmatch}                           
\KwData{$\mathfrak a$  and the absolute factors $\{D_{ij}\}_{\stackrel{i=1,\dots,n-c}{j=1,\dots, s}}$, obtained at Step \ref{factass} of Algorithm \ref{algdec}.}
\KwResult{ After relabeling the polynomials for the index $j$, $L:=[(D_{ij}^{m_j})_{i=1,\dots,n-c})]_{{j=1,\dots, s}}$ with  $V(\mathfrak a+(\sum_{i=1}^n(D_{ij})))$ containing the  component $\C_j$ of $V(\mathfrak a)$ for every $j=1,\dots,s$.}
$L:=$ empty list\;
\For{$j$ from $1$ to $s$}{
Consider a $(n-c)$-uple $(D_{ij})_{i=1,\dots,n-c}$ s.t. $\deg D_{ij}=d_j $ and $m_{ij}=m_j\forall i$\label{n-cupla}\;
Compute the Hilbert Dimension $h$ of the ideal $\mathfrak a+(\sum_{i=1}^n(D_{ij}))$ with $(n-c)$ variables evaluated in 0\;
\eIf{h=0}{
 add the $(n-c)$-uple  $(D_{ij}^{m_j})$ to the list $L$\;}{
go back to Step \ref{n-cupla} and change $(n-c)$-uple\;
}
}
Re-number the factors of $D_i$ in such a way that the $(n-c)$-uples in $L$ are of the form $(D_{1j}^{m_j},\dots D_{n-c, j}^{m_j})$\;
\RETURN $L$
\end{algorithm}


\newpage

\section{Modular Algorithms}\label{compmodp}

We  use the results of Section \ref{sectgoodred} on the exact decomposition strategy  presented in Section \ref{exactstrat}. We develop an algorithm which takes as input an ideal $\mathfrak a$ with generators in $\QQ[\mathbf X]$, defining an equidimensional algebraic set $\C$ in $\CC^n$ of dimension $c$, and gives as output  the number of primary components, their degrees,  multiplicities and the affine Hilbert function of the components of multiplicity 1.

\begin{remark}
We did not present a non-modular version of Algorithm \ref{algpartp}, since all the algorithms presented in Section \ref{exactalgs} are not actually used. We insert this further procedure to avoid useless computations in the calling of Algorithm \ref{algmatchp} in Step \ref{partp} of Algorithm \ref{algdecmodp}.
\end{remark}

\vfill
\

\begin{algorithm}[htbp]
\caption{Modular Algorithm for affine Hilbert Function}         
\label{algdecmodp}                           
\KwData{$\mathfrak a=(F_1,\dots, F_{m})$, $F_i\in \QQ[X_1,\dots, X_n]$, $\mathfrak a$ equidimensional with dimension $c$}
\KwResult{The degree and multiplicity of the primary components of $\mathfrak a$; for primary components of multiplicity 1, their initial ideal with respect to a degree compatible term ordering.}
\PREPROCESSING Perform a generic integer change of coordinates on $\mathfrak a$, with coefficients in $\ZZ$\;\nllabel{preprocp}
 Fix $(n-c)$ different coordinate linear spaces of dimension $c+1$, defined by set of indexes $H_i$, $i=1,\dots, n-c$\;\nllabel{coordspaces}
Compute $D_1(0,\dots,0,X_{c+1})$ as the only generator of $GB_{elim(1)}(\mathfrak a\vert_{X_1=0,\dots,X_c=0})$ respecting (\ref{polyproj})\;
Compute the rational factorization: $D_1(0,\dots,0,X_{c+1})=d_{1}^{(1)}(X_{c+1})^{m_1}\cdots d_{s}^{(1)}(X_{c+1})^{m_s}$\;\nllabel{ratfactp}
\For{$j$ from 1 to $s$}{
Choose a prime integer $p_j$ dividing $d_j^{(1)}(0)$\;\nllabel{choiceofp} 
 Compute $D_i \mod p_j$ as  the only polynomial in $GB_ {elim(i)}(\mathfrak a) \mod p$ respecting (\ref{polyproj}), $i=1,\dots, n-c$\;\nllabel{HilbertThmodp}
Compute the modular factorizations $D_i \mod p_j$, $i=1,\dots,n-c$\;\nllabel{modfactp}
Apply Algorithm \ref{algpartp} to  the rational factor $d_j^{(1)}(X_{c+1})$ and the modular factors of $D_1$, obtaining  the set of modular factors $\mathcal D_j$ of $D_1$\;\nllabel{partp}
Choose $D_{j\tilde k}^{(1)} \mod p_j\in \mathcal D_j$ of minimal degree such that $r_j=\frac{\deg d_j^{(1)}}{\deg D_{j\tilde k}^{(1)}} \in \ZZ$\; \nllabel{scelta}
\If{$r_j\geq 2$}{ 
 Apply Algorithm \ref{algmatchp} to match the modular factor $D_{j\tilde k}^{(1)}$ with $D_{j \tilde k}^{(i)}$ modular factor of  $D_i$, obtaining $\tilde {\mathfrak a}_{i}^{(j)}=\mathfrak a+\sum_{i=1}^{n-c}(D_{j\tilde k}^{(i)})^{m_j} \mod p_j$ (after re-labeling of the factors)\;\nllabel{hdimp}
\If{$m_j$ is 1}{
 Compute the Jacobian matrix of $\mathfrak a\mod p_j$ and  a minor of size $(n-c) \times (n-c)$, $\widetilde M$\;\nllabel{jacobp1}
 Compute $(\tilde {\mathfrak a}_{i}^{(j)}:\widetilde M) \mod p_j$\;\nllabel{satp1}
}
}
}
\RETURN $s$ number of rational components\\
for every $j=1,\dots,s$\\
$r_j$ number of non-rational components constituting the rational component $\mathfrak q_j$\\
$\deg D_{j\tilde k}^{(1)}$ degree of each  non-rational component of $\mathfrak q_j$\\
 $m_j$ multiplicity of the non-rational component\\
 if $m_j=1$, $p_j$ and $(\tilde {\mathfrak a}_{i}^{(j)}:\widetilde M) \mod p_j$ ideal   having the same initial ideal and same Hilbert function as $\mathfrak q_{1j}$\\
 if $m_j\geq 2$, $p$ and $D_{j\tilde k}^{(i)}$, $i=1,\dots,n-c$, image modulo $p$ of $(n-c)$ a polynomial contained in $\mathfrak q_{i}^{j}$ but not in the other absolute components of $\mathfrak a$\;
\end{algorithm}


\begin{algorithm}[htbp]        
\caption{Partition of modular factors}          
\label{algpartp}                           
\KwData  {$d(X_{c+1})\in \QQ[X_{c+1}]$, an integer $p$ dividing $d(0)$ and $D(X_1,\dots,X_{c+1})=\prod_{i=1}^l D_{i}(X_1,\dots,X_{c+1})^{m_i} \mod p$ such that  $d(X_{c+1})\mid D(0,\dots,0,X_{c+1})\mod p$}
\KwResult  {$A$ set containing the modular factors of $d(X_{c+1})\mod p$}
$A:=$ empty list, $i:=1$, $\delta:=1$\;
\While{$i\leq l$}{
$m_i:=$ multiplicity of $d(X_{c+1})$ in the rational factorization of $D(0,\dots,0,X_{c+1})$\;
\If{$D_{i}(0,\dots, 0,X_{c+1})\mod p$ divides $d(X_{c+1}) \mod p$}{ 
add $D_{i}(X_1,\dots,X_{c+1}) \mod p$ to $A$\;
 $\delta=\delta\cdot D_{i}(0,\dots,0,X_{c+1}) \mod p$\;
\If{$\delta=d(X_{c+1}) \mod p$}{
 $i:=l+1$\;
}
}
$i:=i+1$\;
}
\RETURN $A$
\end{algorithm}
\begin{algorithm}[htbp]        
\caption{Matching of modular factors through affine Hilbert Dimension}          
\label{algmatchp}
\KwData{ $D_{i}^{(1)}$ modular factor of $D_1(X_1,\dots, X_{c+1})$ and
 $\{D_{k}^{(j)}\mod p\}_{\stackrel{j=2,\dots, n-c}{k=1,\dots,m}}$ modular factors of $D_j$}
\KwResult {$\tilde {\mathfrak a}_{i}^{j}=\mathfrak a +\sum_{j=1}^{n-c}(D_{i}^{(j)})^{m_i} \mod p$ with Hilbert dimension $c$}
Consider  a $(n-c-1)$-uple $(D_{k}^{(j)})_{j=2,\dots,n-c}$ such that $\deg(D_{k}^{(j)})=\deg(D_{i}^{(1)})$ and $m_k=m_i$ \label{nupla}\;
Compute $h=$Hilbert Dimension of $\mathfrak a +D_{i}^{(1)}+\sum_{j=2}^{n-c}(D_{k}^{(j)})^{m_k} \mod p$ with $(n-c)$ variables evaluated to 0\;
\eIf{$h=0$}{
renumber the modular factor  putting $D_{i}^{(j)}:=D_{k}^{(j)}$}{
go back to Step \ref{nupla} and change $(n-c)$-uple.
} 
 \RETURN $\tilde {\mathfrak a}_{i}^{j}=\mathfrak a +(\sum_{j=1}^{n-c}D_{i}^{(j)})^{m_i} \mod p$.
\end{algorithm}

\newpage

\subsection{Proof of Algorithm \ref{algdecmodp}}
We  apply the results of Section \ref{sectgoodred} to the decomposition strategy explained in Section \ref{exactstrat} and to the Algorithms of Section \ref{exactalgs}. 
Again, we deal with an equidimensional polynomial ideal $\mathfrak a=(F_1,\dots, F_m)$ with dimension $c$, $F_i\in \QQ[\mathbf X]$.
The key point of Algorithm \ref{algdecmodp}  is the choice of a prime integer $p_i$ which gives a ``modular image'' of the algebraic number $\alpha_i$ s.t. $\mathfrak q_i^{(j)}\subseteq \QQ(\alpha_i)[\mathbf X]$.
 For all the notations used, we refer to Section \ref{exactalgs}, Algorithms \ref{algdecmodp}, \ref{algpartp} and \ref{algmatchp}.
 
 \medskip

We will now follow the steps of Algorithm \ref{algdecmodp} in order to show that it gives a correct output.

In Step \ref{preprocp} of Algorithm \ref{algdecmodp}, as in Algorithm \ref{algdec}, we perform a generic change of coordinates; thanks to this, the projections on the "coordinate" linear spaces of dimension $c+1$ are ``generic'' in the sense of Proposition \ref{1to1}: the components of the projected algebraic set  are in one-to-one correspondence with the components of the algebraic set itself (see also Corollary \ref{corproiez}). Furthermore, consider on one hand the absolute factors of the polynomial whose zero set is the projected algebraic set and on the other one the primary components of the ideal defining the algebraic set: factors and primary components are in one-to-one correspondence and the degree and multiplicity of a factor is the degree and multiplicity of the corresponding primary component (in the sense of Definition \ref{Hdeg} and \ref{Hmult}).

In Step \ref{coordspaces} we fix $(n-c)$ distinct ``coordinate'' linear spaces $H_i$ (as explained in Section \ref{exactalgs}). Using projections on these linear spaces, we would like to apply the techniques for absolute factorization developed in \cite{BCG}, but we have to be careful because we do not have one of the main  hypothesis: the Input of the Abs-Fact Algorithm presented in \cite{BCG} is a \emph{rationally irreducible} polynomial. This is not our case, this is why in Step 3 of the algorithm we compute a univariate factorization.\\
Indeed, assume that we are able to compute $D_1$, the only polynomial in the first $c+1$ variables of $GB_{elim(1)}(\mathfrak a)$. This multivariate polynomial in general is not rationally irreducible; furthermore it is not advantageous to compute the multivariate rational factorization of $D_1$.\\
We  rely on Hilbert's Irreducibility Theorem: for infinite integer specialization of $c$ variables, a rationally irreducible factor of the polynomial $D_1$ stays rationally irreducible. This means that if
\[
 D_1(X_1,\dots, X_{c+1})=d_1^{(1)}(X_1,\dots,X_{c+1})^{m_1}\cdots d_s^{(1)}(X_1,\dots,X_{c+1})^{m_s} \in \QQ[X_1,\dots,X_{c+1}]
\]
then for infinite $x_1,\dots,x_c\in \ZZ$ the rational factorization of $D_1(x_1,\dots, x_c,X_{c+1})$ is exactly
\[
 D_1(x_1,\dots, x_c,X_{c+1})=d_1^{(1)}(x_1,\dots, x_c,X_{c+1})^{m_1}\cdots d_s^{(1)}(x_1,\dots, x_c,X_{c+1})^{m_s} \in \QQ[X_{c+1}].
\]

Thanks to the generic change of coordinates of the Preprocessing Step, we can take $x_1=\cdots=x_c=0$. In order to compute this rational univariate factorization without  computing $D_1(X_1,\dots, X_{c+1})$, in Step \ref{HilbertThmodp} we simply specialize $c$ variables of $F_1,\dots, F_m$ and then compute the elimination Groebner Basis:
\[
 D_1(0,\dots, 0,X_{c+1})\in GB_{elim(1)}(\mathfrak a\vert_{X_1=0,\dots,X_c=0}).
\]



Since we are considering a generic projection, a rational factor 
 of $D_1(X_1,\dots,X_{c+1})$ corresponds to  a rational component of the algebraic set $\C=V(\mathfrak a)$ (in the sense of Definition \ref{nonrat}), while each absolute factor of 
  corresponds to an irreducible component.

Once computed in Step \ref{ratfactp} the univariate rational factorization, we then proceed in order to ``break'' the non-rational components.\\
We consider the $j$-th factor of the rational factorization of $D_1$, that is $d_j^{(1)}(X_{c+1})$ which has multiplicity $m_j$. 
If the corresponding factor $d_j^{(1)}(X_1,\dots,X_{c+1})$ of the univariate rational factorization of $D_1(X_1,\dots,X_{c+1})$ is not absolutely irreducible, then its absolute  factors have coefficients in some algebraic extension $\QQ(\alpha_j)$. Using \cite{BCG}, Lemma 11, we can assume that the algebraic extension $\QQ(\alpha_j)$ is generically ge\-ne\-ra\-ted by the evaluation of one absolute factor in a point with integer coordinates. Thanks to the generic change of coordinates, we will choose $(0,\dots,0) \in \ZZ^{c+1}$.\\
We choose an integer prime $p_j$ dividing  $d_j(0)$ (Step \ref{choiceofp}) applying Lemma \ref{rootinQp} and, relying on randomness, we assume that the chosen prime $p_j$ will preserve the initial ideal of the Groebner Basis we will compute along the ``FOR'' loop (as in Lemma \ref{GBmodp}). Thanks to Lemma \ref{rootinQp}, if we factor $D_1(X_1,\dots, X_{c+1})$ modulo this prime $p_j$, the rationally irreducible factor $d_j^{(1)}(X_1,\dots, X_{c+1})$ splits (if it is not absolutely irreducible). The homomorphism $\psi_{p_j}$ of (\ref{hommodp}) is implicitly defined.\\
  Actually we do not compute $D_1(X_1,\dots, X_{c+1})$: in fact, in Step \ref{HilbertThmodp} we compute directly the modular elimination Groebner Basis and then the modular factorizations (Step \ref{modfactp}).
In Step \ref{partp} we group the modular factors corresponding to $d_j^{(1)}(X_{c+1})$ using Algorithm \ref{algpartp}.

If the rational factor $d_j^{(1)}(X_1,\dots,X_{c+1})$ is absolutely irreducible, then it does not further split modulo $p_j$, that is $r_j$, the number of modular factors of $d_j^{(1)}$, is exactly $1$. In this case, we can stop here and repeat the loop for the next rational factor.\\
If $d_j^{(1)}(X_1,\dots,X_{c+1})$ is absolutely reducible, then $r_j\geq 2$ (thanks to the choice of $p_j$ according to Lemma \ref{rootinQp}):  
 in Step \ref{scelta} we choose a modular factor among them having minimal degree which divides $\deg d_j^{(1)}(X_{c+1})$; we assume that this factor is $D_{j\tilde k}^{(1)}(X_1,\dots,X_{c+1})$.\\
 In Step \ref{hdimp} we look for the corresponding modular factor  of $D_i$, $i=2,\dots,n-c$. 
   Using Algorithm \ref{algmatchp}, we obtain the ideal $\tilde {\mathfrak a}_{\tilde k}^{j}=\mathfrak a+\sum_{i=1}^{n-c}(D_{j\tilde k}^{(i)})^{m_j} \mod p_j$ with Hilbert dimension $c$.  Corollary \ref{Hilbdimp} certifies that $\tilde {\mathfrak a}_{\tilde k}^{j}=\psi_{p_j}({\mathfrak a}_{\tilde k}^{j})$.


Once defined in Step \ref{hdimp} the ideal $\tilde {\mathfrak a}_{\tilde k}^{j}$ (re-ordering the indexes) with affine Hilbert dimension $c$, if the  multiplicity $m_j$ is 1, we can keep on following Steps \ref{jacob} and \ref{sat} of Algorithm \ref{algdec}: 
we compute the Jacobian Matrix of $\mathfrak a\mod p_j$  and consider one of its $(n-c)\times (n-c)$-minors, $	\widetilde M$. We compute the colon ideal of $\tilde {\mathfrak a}_{\tilde k}^{j}$ with  $\widetilde M$. Let $M$ be the $(n-c)\times(n-c)$ minor of the Jacobian matrix of $\mathfrak a$ s.t. $M\mod p_j=\widetilde M$.

We need to show that for infinite primes $p_j$ the colon ideal modulo $p_j$ has the same affine Hilbert function of the colon ideal in $\QQ(\alpha_j)[X_1,\dots,X_n]$, that is $\psi_{p_j}({\mathfrak a}_{\tilde k}^{j}:M)=(\tilde {\mathfrak a}_{\tilde k}^{j}:\widetilde M)$.\\
First of all, observe that $\tilde {\mathfrak a}_{\tilde k}^{j}$ and the corresponding non-modular ideal ${\mathfrak a}_{\tilde k}^{j}$ have the same Hilbert function for all but a finite number of prime integers (thanks to Theorem \ref{finitebadphilb}).\\
Furthermore, we can assume that we compute Jacobian matrix of $\mathfrak a$ and a minor  $M$ and then reduce modulo $p_j$. 
For what concerns the colon ideal, the actual computation is performed using a Groebner Basis (see \cite{cox}, Chapter 4, \S 4, Theorem 11 for the details). This means that again we can apply Lemma \ref{GBmodp} and so there is only a finite number of primes $p_j$ such that  $\psi_{p_j}({\mathfrak a}_{\tilde k}^{j}:M)$ and $(\tilde {\mathfrak a}_{\tilde k}^{j}:\widetilde M)$ differ.

\begin{remark}\label{osserv}
Actually, Algorithm \ref{algdecmodp} is a Las-Vegas one, just like the Abs-Fact Algorithm of \cite{BCG}: in fact, in the Preprocessing Step, we have to assume that the coefficients for the generic change of coordinates are taken in a finite set $S\in\ZZ$.

We shall then modify the Preprocessing Step of Algorithm \ref{algdecmodp} and insert a small loop in Step \ref{scelta}, in order to stop the execution and go back to the Preprocessing Step, if we cannot define a $r_j\in \ZZ$ (see Abs-Fact Algorithm of \cite{BCG}).

We can also compute the \lq\lq minimal \rq\rq \ rational algebraic extension $\LL_i=\QQ(\alpha_i)$ containing a set of generators of the ideal $\mathfrak a_i$. We can apply the $LLL$ method developed in \cite{BCG}. Unluckily, we do not have a technique to estimate the needed level of accuracy. We can just try to compute the minimal polynomial which defines $\LL_i$ with increasing levels of accuracy and stop when we get the same polynomial $q(T)$ with 2 different levels of accuracy.
\end{remark}

\section{Tricks on an example}\label{tricks}
The data and Maple files of the examples we are going to discuss are available at\\ https://sites.google.com/site/cristinabertone/examples-for-modular-decomposition

We now test our algorithm on a quite simple example (see the file \emph{DecompositionCIcurveDegree48.mw}).

We consider a complete intersection ideal $\mathfrak a\subseteq\mathbb C[X,Y,Z]$ generated by two polynomials with rational coefficients, $F,G \in \QQ[X,Y,Z]$, of degree 8 each, rationally irreducible.\\
We constructed this c.i. curve in such a way that we know that it has non trivial primary components, in particular it has a rational primary component of degree 14, that splits in 2 absolute primary components of degree 8 each, generated by polynomials in $\QQ(\sqrt{2})[X,Y,Z]$.

The complete intersection curve $\mathcal C=V(\mathfrak a)$ has degree 48 (one can see this, for instance, using a generic plane section and counting points with multiplicity).\\
Since $\mathfrak a$ is generated by 2 polynomials, we can use resultants instead of elimination Groebner Basis to compute the elimination of variables.
We perform a generic linear change of coordinates and we compute 
\[
r:=Res_Z(F(0,Y,Z),G(0,Y,Z)),
\]
which has degree 48 and factors over the rationals (in less than 1 second) as:
\begin{itemize}
\item $d^{(1)}_1(Y)$ factor of degree 14 and multiplicity 1;
\item $d^{(1)}_2(Y)$ factor of degree 4 and multiplicity 1;
\item $d^{(1)}_3(Y)$ factor of degree 22 and multiplicity 1;
\item $d^{(1)}_4(Y)$ factor of degree 2 and multiplicity 2;
\item $d^{(1)}_5(Y)$ factor of degree 1 and multiplicity 4.
\end{itemize}
So, using Definition  \ref{nonrat}, the complete intersection $\mathfrak a$ has 5 rational components $\mathfrak q_i$, three of them with multiplicity 1, with degrees given by $\deg d_i^{(1)}(Y)$ (thanks to Corollary \ref{corproiez}).

\smallskip

We can proceed in the following way: for each rational factor $d^{(1)}_i(Y)$, we choose a prime number $p_i$ dividing $d^{(1)}_i(0)$, except for $i=5$: indeed, we do not look for a prime dividing $d^{(1)}_5(0)$, since this rational component will not further split. We then compute the projections on the coordinate plane modulo $p_i$, $i=1,\dots,4$. Then we compute the modular polynomial describing the projection of the curve for each prime $p_i$ and its modular factorization. We know that $p_i$ forces the rational factor corresponding to $d_i^{(1)}(Y)$ to split (if it is  absolutely reducible). We check whether there is a prime $p_j$ between the chosen ones such that it forces all of the rational factors: if we find one, we can perform all of the computations modulo this prime. If not, we can in any case choose to compute modulo \lq\lq some\rq\rq\  of the primes $p_i$: if $p_i$ and $p_j$ both give the desired splitting for the $i$-th and $j$-th rational factor, then we can compute the corresponding ideals modulo $p_i$ (and not use $p_j$).

For the ideal $\mathfrak a$, we see that $p=89$ give the desired splitting for all of the 4 rational factors which may be absolute reducible. So we will compute only modulo 89.

The computations of $D_1(X,Z)=Res_Y(F,G) \mod p$,  $D_2(X,Y)=Res_Z(F,G) \mod p$  and their modular factorization take less than 4 seconds each.

\smallskip

All of the rational factors further split modulo $p$ in 2 factors. So we compute the initial ideal (and affine Hilbert function) for one of the two absolute components of degrees 7,  2 and 11. The other absolute components have multiplicity $>1$, so we do not perform on them Steps \ref{jacobp1} and \ref{satp1} of Algorithm \ref{algdecmodp}.

Performing Algorithm \ref{algdecmodp}   (including the matching of the factors through Algorithm \ref{algmatchp}), we then obtain the initial ideals:
\begin{itemize}
\item $\deg q_{1}^{(1)}=7$, $in_{tlex}(\mathfrak q_{1}^{(1)})=(X^3, Y^7, X^2Z^2,X^2Y,XY^3,XZ^5XY^2Z^2XYZ^3)$;
\item $\deg q_{2}^{(1)}=2$, $in_{tlex}(\mathfrak q_{2}^{(1)})=(X,Y^2)$;
\item $\deg q_{3}^{(1)}=11$, $in_{tlex}(\mathfrak q_{3}^{(1)})=$\\ \hfill$=(X^3, Y^{11}, XZ^9,X^2Y^2,XY^2,X^2Z^3,XY^4Z,X^2YZ^2,XY^3Z^3,XY^2Z^5,XYZ^7)$.
\end{itemize}

Finally, we use the techniques of the Abs-Fact Algorithm of  \cite{BCG} to compute the polynomial $q(T)\in \ZZ[T]$ which defines the algebraic extension containing the coefficients of a set of generators for the absolute primary components. However we do not have an a priori bound on the size of the coefficients of $q(T)$, as pointed out in Remark \ref{osserv}). We perform an Hensel Lifting of a modular univariate factor until a quite high level of accuracy (in this case, until $p^{512}$); we then construct different candidates for the minimal polynomial, starting with accuracy $p^{16}$, until two different levels of accuracy give the same polynomial.

For accuracy $p^{64}$, we see that the minimal polynomial ``stabilizes'':
\begin{equation*}
q(T)=26301054375\,{T}^{2}-
214355874045600\,T+
+436754388124393216
\end{equation*}
%
%
Obviously, since $\deg q(T)=2$, we can easily find a better presentation of the extension $\QQ(\alpha)$ computing the roots of $q(T)$: we obtain that the extension of $\QQ$ we need can be generated by $\sqrt{2}$.

\medskip

Summing up, we obtained that the complete intersection curve $\mathfrak a=(F,G)\subseteq \QQ[X,Y,Z]$ has the rational primary decomposition
\[
\mathfrak a=\mathfrak q_1\cap\mathfrak q_2\cap\mathfrak q_3\cap\mathfrak q_4\cap\mathfrak q_5
\]
with $\deg \mathfrak q_1=14$, $\deg \mathfrak q_2=4$, $\deg \mathfrak q_3=22$, $\deg \mathfrak q_4=2$ and $\deg \mathfrak q_5=1$ and multiplicities $m_1=m_2=m_3=1$, $m_4=2$, $m_5=4$.

Each of the rational primary ideals with multiplicity 1 further decomposes as
\[
\mathfrak q_i=\mathfrak q_i^{(1)}\cap\mathfrak q_i^{(2)},
\]
with $\mathfrak q_i^{(j)}\subseteq \QQ(\sqrt{2})[X,Y,Z]$, $\mathfrak q_i^{(2)}=\sigma\left(\mathfrak q_i^{(1)}\right)$, where $\sigma(\sqrt{2})=-\sqrt{2}$.

The whole computation took less than 15 minutes on a home-use personal computer, without any problem with memory allocation.

\smallskip

We point out that it is not that obvious to obtain this kind of information about the decomposition of the ideal $\mathfrak a$. For instance, one may use one of the most popular Computer Algebra System, Maple \cite{maple}.\\
We tried to use  the Maple command \texttt{PrimaryDecomposition} (whose algorithm is based on \cite{GTO}), which gives as an output the primary decomposition of the ideal $\mathfrak a$. As input, we also gave the algebraic extension of the rationals in which one can find the generators of the absolute primary decomposition of $\mathfrak a$, namely $\QQ(\sqrt{2})$. Even with this further information about the decomposition (which is not a priori known from the only knowledge of the rational generators of $\mathfrak a$), \texttt{PrimaryDecomposition} in Maple caused a problem with memory allocation (reaching about 2.3 GB), after computing for more than 1 hour. \\
For what concerns other computer algebra systems, we also tried Singular (\cite{sing}), another computer algebra system for polynomial computations. We tried to obtain the rational primary decomposition of $\mathfrak a$ using \texttt{primdecGTZ} and the primary decomposition over $\overline \QQ[X,Y,Z]$ using \texttt{absprimdecGTZ} (which are based on \cite{GTO}, the algorithms are described in \cite{Decksing}). In both cases we stopped the computations after 2 hours, without obtaining the primary decomposition.\\
The CAS \cocoa (\cite{cocoa}) has a command \texttt{PrimaryDecomposition} to decompose   only  monomial square-free ideals. It also has a command called \texttt{EquiIsoDec} which computes an equidimensional isoradical decomposition of $\mathfrak a$, i.e. a list of equidimensional ideals $\mathfrak b_1,\dots,\mathfrak b_k$  such that the radical of $\mathfrak a$ is the intersection of the radicals of $\mathfrak b_1,\dots,\mathfrak b_k$. This command is based on the algorithm presented in \cite{cct} and it works only using $\QQ$ or finite fields as coefficient ring. Nevertheless, \texttt{EquiIsoDec} could not give the output on our example after more than two hours computing.

\smallskip

Although for the moment we cannot really compare our algorithm with the above ones mentioned, we can see that the problem we are facing is challenging and that our modular strategy may move around the computational problems of primary decomposition. Nevertheless, we cannot give complete comparaison for the moment, since we cannot compute the complete primary decomposition nor a reduced decomposition: using Algorithm \ref{algdecmodp} we get several interesting data about the absolute primary components of an equidimensional ideal. Indeed, these data may be useful as a guide or a bound for numerical algorithms, such as the ones in \cite{galrup} or \cite{svw}.

Furthermore, the technique taken from \cite{BCG} to construct the algebraic extension containing the coefficients of a set of generators of a primary component, could be used with other algorithms: for instance, one can see the example \emph{DecompositionCIcurveDegree36.mw}. With our modular strategy, we obtain not only the initial ideal of all of the absolute primary components of a complete intersection in $\QQ[X,Y,Z]$ generated by 2 polynomials both of degree 6, but also that $\QQ(\sqrt{2})$ contains the coefficients of the generators of these components.\\
If we ask Maple to compute the primary decomposition of this ideal, again it does not reach the result after one hour of computation. But if we pass to the command \texttt{PrimaryDecomposition} also the information that the primary components are inside $\QQ(\sqrt{2})[X,Y,Z]$, then we obtain the primary decomposition in less than 3 minutes. 

This example suggests that we may combine at least a part of our strategy (for instance, the construction of the ``splitting'' field) with an existing algorithm (such as the one implemented in Maple), in order to obtain a complete result.

\section*{Conclusions and future work}

In this paper we designed an algorithm which, given  a set of polynomials with rational coefficients defining  an equidimensional ideal $\mathfrak a$, returns the initial ideal of each absolute prime component of $\mathfrak a$. Furthermore it also returns information concerning non-reduced primary components, such as their number, degree, multiplicity. The main ingredients of the algorithm are the classical technique of projection and the use of computations modulo well-chosen primes, as done in \cite{BCG} in order to decompose a bivariate polynomial. The obtained results seem promising, mostly for what concerns complete intersections and more precisely, for curves in $\CC^3$. 

A further step is to implement this algorithm in a Computer algebra System, for instance in Mathemagix \cite{Mathemagix}. Mathemagix is a free computer algebra system under development, which has available libraries for algebraic computation (such as large numbers, polynomials and others) for exact and approximate computation. This should make Mathemagix  particularly suitable as a bridge between symbolic computation and numerical analysis. 

Our final aim is actually to design an algorithm which uses projections with modular techniques and can return the complete absolute primary decomposition of the ideal given as input. The main obstacle to this is the absence of a tool similar to Hensel's Lifting (see for instance \cite{vzg}, Chapter 15, Section 4), which allows lifting a modular factorization to a rational one; we would need a generalization of this in order to lift the modular decomposition of an ideal.

Our next task is then to develop such a tool, design and implement a primary decomposition algorithm and compare its efficiency with other implemented routines. With this tool, we may be able to obtain a complete decomposition for the reduced part of an equidimensional ideal.

Other possible improvements of the algorithm are  dealing with a non-equidimensional ideal $\mathfrak a$ and computing also non-reduced primary components; for the first part we will have to choose whether dealing only with the top-dimensional part of $\mathfrak a$ or  studying also smaller components; for the second one, if we will be able to compute these non-reduced primary components, then we will also be able to compute a radical decomposition (see \cite{cct}).

We are hopeful that our techniques are competitive, since at this moment we can already get a lot of information concerning the absolute decomposition of an ideal, in a reasonable time with a limited use of memory, while other CAS cannot really deal with the primary decomposition of the same ideal (see Section \ref{tricks}).
So, even if at the moment the results of our method are partial and cannot be directly compared to the performances of other softwares, we believe that this method is on the right path to get an efficient primary decomposition algorithm.

\section*{Acknowledgements}

The author is grateful to her Ph.D. advisors Andre Galligo and Margherita Roggero for their constant support and encouragement. The author also wish to thank Gregoire Lecerf for valuable discussions about the subject and the research group Galaad (Inria Sophia Antipolis, France) where she worked during her Ph.D. thesis.

\pagestyle{plain}
\bibliography{biblio3}

\end{document}